\documentclass[12pt,a4paper]{amsart}
\usepackage[bottom]{footmisc}
\usepackage{mathrsfs}
\usepackage{amsmath}
\usepackage{amsthm}
\usepackage{amsfonts}
\usepackage{latexsym}
\usepackage{graphicx}
\usepackage{hyperref}
\usepackage{amssymb}
\usepackage{epsf}
\usepackage{float}
\usepackage{fancyhdr}
\allowdisplaybreaks \makeatletter
\def\rightharpoonfill@{\arrowfill@\relbar\relbar\rightharpoonup}
\DeclareRobustCommand{\overrightharpoon}{\mathpalette{\underarrow@\rightharpoonfill@}}
\makeatother

\allowdisplaybreaks

\begin{document}
\newcommand{\beq}{\begin{equation}}
\newcommand{\eb}{\begin{equation}}
\newcommand{\eneq}{\end{equation}}
\newcommand{\ee}{\end{equation}}
\newtheorem{thm}{Theorem}[section]
\newtheorem{coro}[thm]{Corollary}
\newtheorem{lem}[thm]{Lemma}
\newtheorem{prop}[thm]{Proposition}
\newtheorem{defi}[thm]{Definition}
\newtheorem{rem}[thm]{Remark}
\newtheorem{cl}[thm]{Claim}
\title{Regularity of hyperbolic magnetic Schr\"{o}dinger equation with oscillating coefficients}
\author{Xiaojun Lu$^{1}$\ \ \ \ Xiaofen Lv$^2$}
\thanks{Corresponding author: Xiaojun Lu, Department of Mathematics \& Jiangsu Key Laboratory of
Engineering Mechanics, Southeast University, 210096, Nanjing, China}
\thanks{Keywords: Hyperbolic magnetic Schr\"{o}dinger equation, loss
of regularity, optimality of regularity, microlocal analysis,
instability arguments}
\thanks{Mathematics Subject Classification: 35L10, 35L90, 35S10, 81V10}
\date{}
\maketitle
\pagestyle{fancy}                   % ÉèÖÃҳü
\lhead{X. Lu and X. Lv} \rhead{Regularity of hyperbolic magnetic Schr\"{o}dinger equation} %\rhead{\small\leftmark}
\begin{center}
1. Department of Mathematics \& Jiangsu Key Laboratory of
Engineering Mechanics, Southeast University, 210096, Nanjing,
China\\
2. Jiangsu Testing Center for Quality of Construction Engineering
Co., Ltd, 210028, Nanjing, China
%2. Faculty of Science and Technology, Federation University
%Australia, Ballarat, VIC 3350, Australia
\end{center}
\vspace{.5cm}
\begin{abstract}
This paper mainly discuss the regularity behavior of the hyperbolic
magnetic Schr\"{o}dinger equation with singular coefficients near
the origin. We apply the techniques from the microlocal analysis to
explore the upper bound of loss of regularity. Furthermore, in order
to demonstrate the optimality of the result, a delicate
counterexample with periodic coefficients will be constructed to
show the lower bound of loss of regularity by the application of
harmonic analysis and instability arguments.
\end{abstract}
\section{Introduction to hyperbolic magnetic Schr\"{o}dinger equation and its regularity behavior}
Mathematically speaking, the magnetic field {\bf B} is a solenoidal
vector field whose field line either forms a closed curve or extends
to infinity. In contrast, a field line of the electric field {\bf E}
starts at a positive charge and ends at a negative charge. For
instance, the earth's magnetic field is a consequence of the
movement of convection currents in the outer ferromagnetic liquid of
the core. In the study of quantum mechanics, a magnetic field is
produced by electric fields varying in time, spinning of the
elementary particles, or moving electric charges, etc. Nowadays,
electromagnetic theory is widely utilized in medical research of
organs' biomagnetism, vortex study in the superconductor which
carries quantized magnetic flux, and geographical cataclysms
forecasts, such as
earthquakes, volcanic eruptions, geomagnetic reversal, etc. \cite{CA,LF,MG}.\\

Let ${\bf A}$ be the vector potential of ${\bf B}$, which does not
depend on time, ${\bf B}=\nabla\times{\bf A}.$ Clearly, $
\nabla\cdot{\bf B}={\rm div}\ {\rm rot}{\bf A}=0.$ We deduce from
the Maxwell's equation ($\mu$ is the magnetic permeability)
$\nabla\times{\bf E}=-\mu{\partial{\bf B}}/{\partial t}=0$ that
${\bf E}=-\nabla\phi,$ where the scalar $\phi$ represents the
electric potential. Next we choose an appropriate Lagrangian for the
charged particle in the electromagnetic field ($q$ is the electric
charge of the particle, and ${\bf v}$ is its velocity, $m$ is mass)
$\mathscr{L}={m{\bf v}^2}/{2}-q\phi+q{\bf v}\cdot{\bf A}.$ The
canonical momentum is specified by the equation ${\bf p}=\nabla_{\bf
v}\mathscr{L}=m{\bf v}+q{\bf A}.$ Then we define the classical
Hamiltonian by Legendre transform, $\mathscr{H}:={\bf p}\cdot{\bf
v}-\mathscr{L}={({\bf p}-q{\bf A})^2}/{(2m)}+q\phi.$ In quantum
mechanics, we replace ${\bf p}$ by $-i\hbar\nabla$ ($\hbar$ is the
Planck constant), $\mathscr{H}={(i\hbar\nabla+q{\bf
A})^2}/{(2m)}+q\phi.$ This Hamiltonian operator phenomenologically
describes a quantity of behaviors discovered in superconductors and
quantum electrodynamics(QED). Ginzburg-Landau equations,
Schr\"{o}dinger equations, Dirac equations and the matrix Pauli
operator are famous examples in these
respects \cite{HK,ME,MR}.\\

To be more specific, our mathematical model arises from the
discussion of the extrema of the following variational problem
$I[u]$ in the study of quantum mechanics \cite{L1, MT1, MT3},
($u\in\mathcal{U}$ will be explained in Section 2),
$$\displaystyle I[u]:={1}/{2}\int_\Omega\Big(|u_t|^2-b^2(t)|(i\nabla+{\bf A}(x))u|^2-\phi(x)|u|^2\Big)dx.$$
Let the Lagrangian be
$$\mathcal{L}(t,x_1,\cdots,x_n,u,\bar{u},u_t,\bar{u}_t,\nabla u,\nabla \bar{u})\triangleq|u_t|^2-b^2(t)|(i\nabla+{\bf A}(x))u|^2-\phi(x)|u|^2.$$
The Euler-Lagrangian equation for $\mathcal{L}$ is of the form
$$\frac{\partial\mathcal{L}}{\partial u}-\frac{\partial}{\partial t}\Big(\frac{\partial\mathcal{L}}{\partial u_t}\Big)-\displaystyle\sum_{i=1}^n\frac{\partial}{\partial x_i}\Big(\frac{\partial\mathcal{L}}{\partial u_{x_i}}\Big)=0.$$
In fact, simple calculation leads to
$$\frac{\partial\mathcal{L}}{\partial u}=b^2(t)\Big(i{\bf
A}\cdot\nabla\bar{u}-{\bf A}^2\bar{u}-\phi(x)\bar{u}\Big),\ \
\frac{\partial\mathcal{L}}{\partial u_t}=\bar{u}_t,\ \
\frac{\partial\mathcal{L}}{\partial
u_{x_i}}=b^2(t)\Big(-\bar{u}_{x_i}-ia_i\bar{u}\Big).$$ Consequently,
one has
$$u_{tt}+b^2(t)(i\nabla+{\bf
A})^2u+\phi(x)u=0.$$

As is known, loss of regularity is an essential topic when we study
the well-posedness of partial differential equations. For instance,
there is no loss for Cauchy problem of the classical wave operator
and Klein-Gordon operator. In other words, for a sufficiently large
Sobolev index $s$, when $(u_0,u_1)\in H^{s+1}\times H^{s}$ are
given, then there exists a unique solution $u$ belonging to the
following function spaces: $ C([0,\infty);H^{s+1})\bigcap
C^1([0,\infty);H^{s}).$ For the classical wave equation with
variable coefficients, \cite{FC9,L3,REI1} introduced a
classification of regularity behaviors incurred by the singular
coefficients. As for the difference of regularity for initial Cauchy
data, \cite{MC3} considered the typical $p$-evolution model in
1-Dimension. In effect, the principal operator determined the
difference $p$ when its coefficient
is Log-Lipschitz continuous with respect to the time \cite{RA2,MC4,MC5}.\\

Let $\Omega\subset\mathbb{R}^N$ be a bounded open set with a
time-independent vector potential ${\bf A}\in
(C^1(\overline{\Omega}))^N$, its boundary $\Gamma\in C^2$. Assume
that $\mathcal{H}(\Omega)$ is a Hilbert space. From the Hamiltonian,
we can define the corresponding vector operator
$$\mathscr{H}_{\bf A}:= i\nabla+{\bf
A}(x):\mathcal{H}(\Omega)\to (\mathcal{H}(\Omega))^N.$$ The
concerned function spaces will be detailed in the next section. In
this work, we address the regularity behavior of the hyperbolic
magnetic Schr\"{o}dinger equation without influence from the
electric field ${\bf E}$,
\beq\left\{\begin{array}{ccc} u_{tt}+b^2(t)\mathscr{H}_{\bf A}^2u=0&(t,x)\in(0,T)\times\Omega&\\
u=0 &(t,x)\in(0,T)\times\Gamma\\
u(0,x)=u_0(x),u_t(0,x)=u_1(x)&x\in\Omega.&
\end{array}\right.\eneq
And the time-dependent oscillating coefficient $b$ satisfies the
following assumptions:
\begin{itemize}
\item ({\bf Assumption I}): $\displaystyle C_1\leq\inf_{t\in(0,T]}b(t)\leq\sup_{t\in(0,T]}b(t)\leq C_2$,
$C_1,C_2>0$;
\item ({\bf Assumption II}): $b\in C^2(0,T]$, $\displaystyle \Big|\frac{d^{k}}{dt^k}b(t)\Big|\leq
C_3(\nu(t)/t)^{k}$ uniformly for $k=1,2$, where $\nu\in C(0,T]$ is a
positive, strictly decreasing function satisfying
$\displaystyle\inf_{t\in(0,T]}\nu(t)\geq C_4>0.$
\end{itemize}

Under the above assumptions, we are ready to give the main result
concerned with the regularity of the Cauchy problem (1).
\begin{thm}
Define $\mu(t)$ as $t/\nu(t)$. For the Sobolev index $s\geq1$, if
the oscillating coefficient $b$ satisfies Assumptions I and II,
$(u_0,u_1)\in H^s\times H^{s-1}$, then, there exists a unique
solution $u$ belonging to the following function spaces: $$ u\in
C\Big([0,T]; \exp\Big(c_1\nu(\mu^{-1}(2^P/\sqrt{(i\nabla+{\bf
A}(x))^2}))\Big)H^s\Big),$$
$$u_t\in C\Big([0,T]; \exp\Big(c_1\nu(\mu^{-1}(2^P/\sqrt{(i\nabla+{\bf
A}(x))^2}))\Big)H^{s-1}\Big),$$ where $P$ is a fixed appropriate
positive integer and $c_1$ is a positive constant. Moreover,
$\mu^{-1}$ denotes the inverse function of $\mu$.
\end{thm}
\begin{rem}
It is clear that, the oscillating coefficient incurs loss of
regularity. Actually, the above theorem also holds when we replace
$2^P$ by a general sufficiently large positive number. One chooses
$2^P$ since the contrast is very significant while comparing with
the coefficient sequences constructed in the counterexample in
Section 4. Moreover, it is worth noticing that if
$\nu(t)=(\log(1/t))^\gamma$, $\gamma\in (0,1)$, increasing $P$
extends the Sobolev spaces.
\end{rem}
\begin{rem}
The function $\mu^{-1}$ is uniquely determined since $\mu$ is a
strictly increasing function in $(0,T]$. Particularly,  when
$\nu(t)\leq C$, there is no loss of regularity. This theorem shows
explicitly the so-called (at most) $\nu$-loss of regularity which
arises from the singular coefficient near the origin.
\end{rem}
\begin{rem}
We give some typical examples to explain the different influence
from various kinds of oscillating coefficients. Suppose that
$(u_0,u_1)\in H^s\times H^{s-1}$, then according to Theorem 1.1,
there exists a unique solution $u$ belonging to the following
function spaces ($\alpha,\beta,\gamma>0$):
\begin{enumerate}
\item $\nu(t)\sim1$, no loss of regularity,
$$u\in C([0,T];H^s),\ \ \ u_t\in C([0,T];H^{s-1}).$$
\item $\nu(t)=\log({1}/{t})$, finite loss of regularity,
$$u\in C([0,T];H^{s-\alpha}),\ \ \ u_t\in
C([0,T];H^{s-\alpha-1}).$$
\item $\nu(t)=\Big(\log({1}/{t})\Big)^{\gamma}$, $\gamma\in(0,1)$, arbitrarily small loss of regularity,
$$ u\in C\Big([0,T];(\sqrt{(i\nabla+{\bf
A}(x))^2})^{\beta\Big(\log({\sqrt{(i\nabla+{\bf
A}(x))^2}}/{2^{P}})\Big)^{\gamma-1}}H^s\Big),$$
$$u_t\in
C\Big([0,T];(\sqrt{(i\nabla+{\bf
A}(x))^2})^{\beta\Big(\log({\sqrt{(i\nabla+{\bf
A}(x))^2}}/{2^{P}})\Big)^{\gamma-1}}H^{s-1}\Big).$$
\item $\nu(t)=\Big(\log({1}/{t})\Big)\Big(\log^{[2]}({1}/{t})\Big)^{\gamma_2}\cdots\Big(\log^{[n]}({1}/{t})\Big)^{\gamma_n}$,
$\gamma_i\in(0,1]$, $i=2,\cdots,n$, infinite loss of regularity,
$$u\in C\Big([0,T];(\sqrt{(i\nabla+{\bf
A}(x))^2})^{\gamma\Big(\log^{[2]}({\sqrt{(i\nabla+{\bf
A}(x))^2}}/{2^{P}})\Big)^{\gamma_2}\cdots\Big(\log^{[n]}({\sqrt{(i\nabla+{\bf
A}(x))^2}}/{2^{P}})\Big)^{\gamma_n}}H^s\Big),$$
$$u_t\in C\Big([0,T];(\sqrt{(i\nabla+{\bf
A}(x))^2})^{\gamma\Big(\log^{[2]}({\sqrt{(i\nabla+{\bf
A}(x))^2}}/{2^{P}})\Big)^{\gamma_2}\cdots\Big(\log^{[n]}({\sqrt{(i\nabla+{\bf
A}(x))^2}}/{2^{P}})\Big)^{\gamma_n}}H^{s-1}\Big).$$
\end{enumerate}
\end{rem}
The rest of the paper is organized as follows. Section 2 is devoted
to the description of $\mathscr{H}_{\bf A}^2$-induced function
spaces $\mathcal{H}_0^1$, $\mathcal{H}^{-1}$ and introduction of
$\mathscr{H}_{\bf A}^2$-pseudodifferential operators. And the usual
compactness-uniqueness argument is applied to demonstrate a
generalized Poincar\'{e}'s inequality. In Section 3, we use some
powerful tools, such as symbol calculus, normal diagonalisation,
etc. from micro-local analysis and WKB analysis to obtain the
precise $\nu$-loss of regularity. In Section 4, we discuss the
optimality of the loss of regularity by the application of harmonic
analysis and instability arguments.
\section{Prerequisites: basic functional spaces and pseudodifferential operators}
\subsection{$\mathscr{H}_{\bf A}^2$-induced Hilbert spaces}
Now we give a function space induced by the vector operator
$\mathscr{H}_{\bf A}$.
\begin{defi} Let ${\bf A}\in (L^\infty(\Omega))^N$, and we define a complex function space
$$\mathcal{H}^{1}(\Omega):=\{\omega:\omega\in
L^2(\Omega),\mathscr{H}_{\bf A}\omega\in (L^2(\Omega))^N\},$$ which
is equipped with the norm
$$\|\omega\|_{\mathcal{H}^1}:=\sqrt{\|\omega\|_{L^2}^2+\|\mathscr{H}_{\bf
A}\omega\|_{(L^2)^{N}}^2},$$ where
$$\|(\omega_1,\cdots,\omega_N)\|_{(L^2)^N}:= \sqrt{\displaystyle\sum_{\ell=1}^N\|\omega_\ell\|_{L^2}^2}.$$
One defines $\mathcal{H}_0^1$ as the closure of
$\mathscr{D}(\Omega)$ in $\mathcal{H}^1$, and $\mathcal{H}^{-1}$ as
the dual space of $\mathcal{H}_0^1$.
\end{defi}
\begin{lem} Actually, $\mathcal{H}^1$ is an equivalent
definition of the Sobolev space $H^1$. Consequently, the imbeddings
$\mathcal{H}_0^1\hookrightarrow L^2$ and
$L^2\hookrightarrow\mathcal{H}^{-1}$ are both dense and compact.
\end{lem}
\begin{proof} Indeed, utilizing the definition of norm in each space, one has
\begin{itemize}
\item $\mathcal{H}^1\hookrightarrow H^1$
$$\begin{array}{lll}
\|\omega\|_{H^1}^2&=&\|\omega\|_{L^2}^2+\|\nabla\omega\|_{(L^2)^N}^2\\
\\
&=&\|\omega\|_{L^2}^2+\|\mathscr{H}_{\bf A}\omega-{\bf A}\omega\|_{(L^2)^N}^2\\
\\
&\leq&\|\omega\|_{L^2}^2+2\|\mathscr{H}_{\bf A}\omega\|_{(L^2)^N}^2+2\|{\bf A}\omega\|_{(L^2)^N}^2\\
\\
&\leq&(1+2N\|{\bf A}\|_{L^\infty}^2)\|\omega\|_{L^2}^2+2\|\mathscr{H}_{\bf A}\omega\|_{(L^2)^N}^2.\\
\\
\end{array}$$
\item $H^1\hookrightarrow \mathcal{H}^1$
$$\begin{array}{lll}
\|\omega\|_{\mathcal{H}^1}^2&=&\|\omega\|_{L^2}^2+\|\mathscr{H}_{\bf A}\omega\|_{(L^2)^N}^2\\
\\
&=&\|\omega\|_{L^2}^2+\|(i\nabla+{\bf A}(x))\omega\|_{(L^2)^N}^2\\
\\
&\leq&\|\omega\|_{L^2}^2+2\|\nabla\omega\|_{(L^2)^N}^2+2\|{\bf A}(x)\omega\|_{(L^2)^N}^2\\
\\
&\leq&(1+2N\|{\bf A}\|_{L^\infty}^2)\|\omega\|_{L^2}^2+2\|\nabla\omega\|_{(L^2)^N}^2.\\
\\
\end{array}$$
\end{itemize}
\end{proof}

To introduce the pseudodifferential operators, first, we give a
series of generalized Green's formulas for the second order operator
$(i\mathscr{H}_{\bf A})^2$ on $H^2$.
\begin{lem} For $u,v\in H^2$, $\Gamma\in C^2$, one has $$\int_{\Omega}(i\mathscr{H}_{\bf
A})^2u\bar{v}dx=\int_\Gamma\frac{\partial
u}{\partial\nu_{i\mathscr{H}_{\bf
A}}}\cdot\bar{v}d\Gamma-\int_{\Omega}i\mathscr{H}_{\bf
A}u\cdot\overline{i\mathscr{H}_{\bf A}v} dx,$$ $$
\int_{\Omega}(i\mathscr{H}_{\bf
A})^2u\bar{v}dx-\int_{\Omega}u\overline{(i\mathscr{H}_{\bf
A})^2v}dx=\int_\Gamma\frac{\partial
u}{\partial\nu_{i\mathscr{H}_{\bf
A}}}\cdot\bar{v}d\Gamma-\int_\Gamma u\cdot\overline{\frac{\partial
v}{\partial\nu_{i\mathscr{H}_{\bf A}}}}d\Gamma,$$
$$\int_{\Omega}(i\mathscr{H}_{\bf A})^2udx=\int_\Gamma\frac{\partial
u}{\partial\nu_{i\mathscr{H}_{\bf A}}}d\Gamma-\int_{\Omega}{\bf
A}(x)\cdot\mathscr{H}_{\bf A}u dx,$$ where \beq\frac{\partial
}{\partial\nu_{i\mathscr{H}_{\bf A}}}:= (\nabla-i{\bf
A})\cdot\nu=\frac{\partial }{\partial\nu}-i{\bf A}(x)\cdot\nu, \eneq
and $\nu$ is the unit outward normal.
\end{lem}

\begin{proof}
Keep in mind the classical trace theory in \cite{RA}. If $\Omega$ is
bounded and $\Gamma\in C^2$, then $\mathscr{D}(\overline{\Omega})$
is dense in $H^2$. And the trace mapping
$v\mapsto\overrightarrow{\gamma}v=(\gamma_0v,\gamma_1v)=(v|_\Gamma,\frac{\partial
v}{\partial\nu}\Big|_\Gamma)$ from $H^2(\Omega)$ to
$H^{{3}/{2}}(\Gamma)\times H^{{1}/{2}}(\Gamma)$ is linear and
continuous. So we prove these identities on
$\mathscr{D}(\overline{\Omega})$. On the one hand,
$$\begin{array}{lll} &&\displaystyle\int_\Omega i\mathscr{H}_{\bf
A}u\cdot\overline{i\mathscr{H}_{\bf
A}v}dx\\
\\
&=&\displaystyle\int_\Omega \nabla u\cdot\nabla\bar{v}dx-\int_\Omega
i{\bf A}(x)u\cdot\nabla\bar{v}dx+\int_\Omega i\nabla u\cdot{\bf
A}(x)\bar{v}dx+\int_\Omega{\bf A}{\bf
A}^Tu\bar{v}dx\\
\\
&=&\displaystyle\int_\Omega \nabla u\cdot\nabla\bar{v}dx+\langle
u,i{\bf A}\cdot\nabla v\rangle_{L^2}+\int_\Gamma i
u\bar{v}\cdot({\bf A}\cdot\nu)d\Gamma+\langle u,i\nabla\cdot {\bf A}
v\rangle_{L^2}+\displaystyle\int_\Omega{\bf A}{\bf A}^Tu\bar{v}dx.
\end{array}$$
On the other hand,
$$\begin{array}{lll} &&
\displaystyle\int_{\Omega}(i\mathscr{H}_{\bf
A})^2u\bar{v}dx=\displaystyle\int_\Gamma\frac{\partial
u}{\partial\nu}\cdot\bar{v}d\Gamma-2\int_\Gamma iu\bar{v}\cdot({\bf
A}(x)\cdot\nu)d\Gamma-\displaystyle\int_\Omega \nabla
u\cdot\nabla\bar{v}dx\\
\\
&&\displaystyle-\langle u,i{\bf A}\cdot\nabla v\rangle_{L^2}-\langle
u,i\nabla\cdot{\bf A} v\rangle_{L^2}-\int_\Omega{\bf A}{\bf
A}^Tu\bar{v}dx.\\
\\
\end{array}$$
Notice the definition (2), and one concludes the proof of the first
identity. The second identity follows immediately when we consider
the conjugate of the first identity. Finally, the third identity is
the special case of $v\equiv1$ of the first identity.
\end{proof}
\begin{rem}
When ${\bf A}\equiv0$, one has the classical Green's formulas for
Laplacian $\Delta$. From the above lemma, we know that
$\mathscr{H}_{\bf A}^2$ is a self-adjoint differential operator on
$H^2\bigcap H_0^1$. In this case, (2) becomes the usual unit outward
normal derivative.
\end{rem}
In the following, one introduces a considerably significant result,
which plays a crucial part in the description of the dual of
$\mathcal{H}_0^1$.
\begin{lem}(Generalized
Poincar\'{e}'s inequality) Let $\Gamma\in C^1$. Then for any
$\omega\in \mathcal{H}_0^1$, there is a constant $C(\Omega)>0$ such
that $$\|\omega\|_{L^2}\leq C(\Omega)\|(i\nabla+{\bf A})
\omega\|_{(L^2)^N}.$$
\end{lem}
\begin{proof}
Assume that
$$\Omega\subset[x_{01},x_{11}]\times[x_{02},x_{12}]\times\cdots\times[x_{0N},x_{1N}].$$
Let us define a semi-norm on $\mathcal{H}^1$, i.e.
$$|\omega|_{\dot{\mathcal{H}}^1}:=\|(i\nabla+{\bf A}) \omega\|_{(L^2)^N}=\sqrt{\displaystyle\sum_{j=1}^{N}\|i\frac{\partial \omega}{\partial x_j}+a_j(x)\omega\|_{L^2}^2}.$$
First we prove that the above semi-norm is actually a norm on
$\mathcal{H}^1_0$. In fact, let $\omega$ be a function from
$\mathcal{H}_0^1(\Omega)$ such that
$|\omega|_{\dot{\mathcal{H}^1}}=0$. Then one has a system of
ordinary differential equations in $\Omega$,
$$\forall j=1,\cdots,N,\ \ i\frac{\partial \omega}{\partial x_j}+a_j(x)\omega=0.$$

Let $\nabla_j$ denote $\frac{\partial}{\partial x_j}$. Here we
introduce an important matrix, i.e. the compatibility matrix
$\Xi_{\bf A}$,
$$\Xi_{\bf A}:=\left(\begin{array}{cccc}\xi_{11}&\xi_{12},&\cdots&\xi_{1N}\\
\xi_{21}&\xi_{22}&\cdots&\xi_{2N}\\
\vdots&\vdots&\cdots&\vdots\\
\xi_{N1}&\xi_{N2}&\cdots&\xi_{NN}\\
\end{array}\right) $$
where
$$\xi_{jk}:=\left|\begin{array}{cc}\nabla_j&\nabla_k\\
a_j&a_k\end{array}\right|.$$ Clearly, $\Xi_{\bf A}$ is an
antisymmetric matrix. In quantum mechanics, $\Xi_{\bf A}\equiv0$
stands for the case without magnetic field, i.e. $${\bf B}={\rm
rot}{\bf A}=0.$$ Once the magnetic field exists, then $\Xi_{\bf
A}\neq0$. Consequently, $\Xi_{\bf A}$ serves as a test matrix for
the magnetic field. If $\Xi_{\bf A}\neq0$, i.e. the magnetic field
exists, then the above system has only the trivial solution
$\omega=0$. Otherwise, if $\Xi_{\bf A}=0$, then there exists a
unique solution represented by $$\
\omega=C_0\exp\Big(i\int_{x_{0k}}^{x_k}a_k(x_1,\cdots,x_{k-1},s_k,x_{k+1},\cdots,x_N)ds_k\Big),\
\ \ \forall k=1,\text{or}\ \cdots,N,$$ where $C_0$ is any complex
constant. Indeed, for $j\neq k$, when $i\neq j$ and $i\neq k$, then
one has
$$\displaystyle \frac{\partial}{\partial x_i}\Big(\int_{x_{0j}}^{x_j}a_jds_j-\int_{x_{0k}}^{x_k}a_kds_k\Big)=0.$$
While for $j\neq k$, when $i$ is equal to any one of them, then we
have the same conclusion. As a result, for $\forall j,k=1,\cdots,N$,
$$\int_{x_{0j}}^{x_j}a_jds_j-\int_{x_{0k}}^{x_k}a_kds_k\equiv C.$$ Immediately, one deducts that $|\omega|$ is a constant in $\Omega$. Notice the fact $\omega\in
\mathcal{H}_0^1$, then one has $\omega=0$ in $\Omega$.\\

Define an equivalent norm in $\mathcal{H}^1$, i.e.
$$\|\omega\|_{\mathcal{H}^1}:= \|\omega\|_{L^2}+|\omega|_{\dot{\mathcal{H}}^1}.$$ We
prove the inequality by contradiction. If there does not exist any
such constant $C(\Omega)$ such that,
$\forall\omega\in\mathcal{H}_0^1$,
$$\|\omega\|_{L^2}\leq C(\Omega)|\omega|_{\dot{\mathcal{H}}^1},$$ then one can
find a sequence $\{\omega_m\}_m$ from $\mathcal{H}_0^1$ such that
$${1}/{m}\|w_m\|_{L^2}>|w_m|_{\dot{\mathcal{H}}^1}.$$ Let
$$v_m\triangleq{w_m}/{\|w_m\|_{\mathcal{H}^1}},$$ then one defines a
sequence $\{v_m\}_m$ from $\mathcal{H}^1_0$ such that
\beq\|v_m\|_{\mathcal{H}^1}=1,\eneq \beq
|v_m|_{\dot{\mathcal{H}}^1}<{1}/{m}.\eneq Since $\Omega$ is bounded
and open, $\Gamma\in C^1$, then the canonic injection from
$\mathcal{H}^1(\Omega)$ to $L^2(\Omega)$ is compact. From $(3)$, one
can extract a subsequence $\{v_\mu\}_\mu$ from the sequence
$\{v_m\}_m$ such that $$v_\mu\to v\ \ \ \text{in}\ \ L^2(\Omega).$$
From $(4)$, one has $$\forall j=1,\cdots,N,\ \ i\frac{\partial
v_\mu}{\partial x_j}+a_j(x)v_\mu\to0\ \ \ \text{in}\ \
L^2(\Omega).$$ Since $\mathcal{H}_0^1$ is complete, then one obtains
$$v_\mu\to v\ \ \ \text{in}\ \ \mathcal{H}_0^1,$$ with
$$\forall j=1,\cdots,N,\ \
i\frac{\partial v}{\partial x_j}+a_j(x)v=0.$$ Accordingly, one has
$$v\equiv0.$$ This is impossible since
$$\|v\|_{\mathcal{H}^1}=\displaystyle\lim_{\mu\to\infty}\|v_\mu\|_{\mathcal{H}^1}=1.$$
This concludes the proof.
\end{proof}
\subsection{$\mathscr{H}_{\bf A}^2$-induced pseudodifferential operators}
According to Lemma 2.5, we introduce an equivalent norm in
$\mathcal{H}_0^1$,
$$\|u\|_{\mathcal{H}^1_0}:=\|\mathscr{H}_{\bf
A}u\|_{(L^2)^N},\ \ \forall u\in \mathcal{H}^1_0.$$ Lemma 2.3
indicates $$(\mathscr{H}_{\bf
A}^2u,v)_{L^2}=(u,v)_{\mathcal{H}^1_0},\ \ \ \forall u\in
\mathcal{H}_0^1\ \ \ \text{such that}\ \mathscr{H}_{\bf A}^2u\in
L^2,\ \ \forall v\in \mathcal{H}_0^1.$$ As is discussed, the
imbeddings $\mathcal{H}_0^1\hookrightarrow L^2$ and
$L^2\hookrightarrow\mathcal{H}^{-1}$ are both dense and compact.
Consequently, $\mathcal{H}_0^1\hookrightarrow \mathcal{H}^{-1}$ is
dense and compact. As a result, it is reasonable to introduce the
duality mapping
$$\mathscr{H}_{\bf A}^2:\mathcal{H}_0^1\to\mathcal{H}^{-1}$$ defined
by $$\langle\mathscr{H}_{\bf
A}^2u,v\rangle_{\mathcal{H}^{-1},\mathcal{H}^1_0}:=(u,v)_{\mathcal{H}^1_0},\
\forall u,v\in \mathcal{H}^1_0. $$ By Riesz-Fr\'{e}chet
representation theorem, it holds that $\mathscr{H}_{\bf A}^2$ is an
isometric isomorphism of $\mathcal{H}_0^1$ onto $\mathcal{H}^{-1}$.
This indicates, $\mathscr{D}(\Omega)$ is also dense in
$\mathcal{H}^{-1}$. Denoting the compact imbedding $$ I:
\mathcal{H}_0^1\to\mathcal{H}^{-1}.$$ Then we define a linear and
compact mapping $$ S\triangleq (\mathscr{H}_{\bf A}^2)^{-1}\circ I:
\mathcal{H}_0^1\to\mathcal{H}_0^1.$$ Furthermore, $S$ is positive
and self-adjoint. Indeed, for $\forall u,v\in\mathcal{H}_0^1$, on
the one hand,
$$(Su,v)_{\mathcal{H}^1_0}=((\mathscr{H}_{\bf
A}^2)^{-1}u,v)_{\mathcal{H}^1_0}=(u,v)_{L^2}.$$ On the other hand,
$$(u,Sv)_{\mathcal{H}^1_0}=(u,(\mathscr{H}_{\bf
A}^2)^{-1}v)_{\mathcal{H}^1_0}=(u,v)_{L^2}.$$ Hence, $$
(Su,v)_{\mathcal{H}^1_0}=(u,Sv)_{\mathcal{H}^1_0}.$$ Applying the
spectral theorem in \cite{MT1}\cite{MT3} to the compact,
self-adjoint and positive linear operator $S$, we conclude that the
spectrum for $\mathscr{H}_{\bf A}^2$ on $\mathcal{H}_0^1$ is
discrete. Here we denote as $\Lambda:=\{\lambda^2_k\}_k$. And the
point spectrum satisfies
$$ 0<\lambda^2_1\leq\lambda^2_2\leq\lambda^2_3\leq\cdots\to\infty,$$ with
finite multiplicity. In addition, there exists an orthogonal system
of complex-valued eigenfunctions
$\{\phi_\lambda(x)\}_{\lambda^2\in\Lambda}$ in $\mathcal{H}^1_0$,
and for each $\lambda^2\in\Lambda$,
$$\|\phi_\lambda\|_{L^2}=1.$$ More importantly, $\{\phi_\lambda(x)\}_{\lambda^2\in\Lambda}$ is dense in $\mathcal{H}_0^1$. Hereafter, we denote by $Z$
the finite combinations of eigenfunctions $\phi_\lambda$. Obviously,
$Z$ is dense in $\mathcal{H}_0^1$.

\begin{rem}
$\{\phi_\lambda(x)\}_{\lambda^2\in\Lambda}$ has orthogonality in
both $L^2$ and $\mathcal{H}^{-1}$. Indeed, for $k\neq l$,$$
0=(\phi_{\lambda_k},\phi_{\lambda_l})_{\mathcal{H}_0^1}=\langle\mathscr{H}_{\bf
A}^2\phi_{\lambda_k},\phi_{\lambda_l}\rangle_{\mathcal{H}^{-1},\mathcal{H}_0^1}=\lambda^2_k(\phi_{\lambda_k},\phi_{\lambda_l})_{L^2}.
$$ While by using the isometric property of $\mathscr{H}_{\bf A}^2$
on $\mathcal{H}_0^1$, one has $$
(\phi_{\lambda_k},\phi_{\lambda_l})_{\mathcal{H}^{-1}}=((\mathscr{H}_{\bf
A}^2)^{-1}\phi_{\lambda_k},(\mathscr{H}_{\bf
A}^2)^{-1}\phi_{\lambda_l})_{\mathcal{H}_0^1}={\lambda_k}^{-2}{\lambda_l}^{-2}(\phi_{\lambda_k},\phi_{\lambda_l})_{\mathcal{H}^{1}_0}=0.$$
\end{rem}
\begin{rem}
Moreover, $\{\phi_\lambda(x)\}_{\lambda^2\in\Lambda}$ is also dense
in both $L^2$ and $\mathcal{H}^{-1}$ since the density of the
imbeddings $$\mathcal{H}_0^1\hookrightarrow
L^2\hookrightarrow\mathcal{H}^{-1}.$$
\end{rem}
\begin{rem}
Let $\Omega=(0,\pi)$, for the Dirichlet operator
$$(i\frac{d}{d x}-1)^2: \mathcal{H}_0^1\to
\mathcal{H}^{-1},$$ it is easy to calculate that
$\{1,2^2,3^2,\cdots,N^2,\cdots\}$ is the set of eigenvalues which
are bounded away from 0. And the associated orthonormal basis (in
the sense of $L^2$-norm) in $\mathcal{H}_0^1$ is
$$\Big\{\sqrt{{2}/{\pi}}\sin(x)e^{-ix},\sqrt{{2}/{\pi}}\sin(2x)e^{-ix},\sqrt{{2}/{\pi}}\sin(3x)e^{-ix},\cdots,\sqrt{{2}/{\pi}}\sin(Nx)e^{-ix},\cdots\Big\}.$$
\end{rem}

With the above notations, one can define the generalized Fourier
transform for $f\in\mathcal{H}^{-1}$ as follows: \beq
\hat{f}(\lambda):=\langle
f,\phi_{\lambda}\rangle_{\mathcal{H}^{-1},\mathcal{H}_0^1}.\eneq And
the corresponding Fourier series, a unique orthogonal expansion in
$\mathcal{H}^{-1}$, is of the form \beq
f(x)=\sum_{\lambda^2\in\Lambda}\hat{f}(\lambda)\phi_\lambda(x),\eneq
with the RHS converging in $\mathcal{H}^{-1}$. Indeed, for $\forall
f\in\mathcal{H}^{-1}$, there is a unique $u_f\in\mathcal{H}^1_0$
such that $\langle
f,v\rangle_{\mathcal{H}^{-1},\mathcal{H}_0^1}=(u_f,v)_{\mathcal{H}_0^1}$
for $\forall v\in\mathcal{H}_0^1$. Then
$$\sum_{\lambda\in\Lambda}|\hat{f}(\lambda)|^2\|\phi_\lambda\|_{\mathcal{H}^{-1}}^2=\sum_{\lambda\in\Lambda}|(u_f,\phi_\lambda)_{\mathcal{H}_0^1}|^2\|\phi_\lambda\|_{\mathcal{H}^{-1}}^2=\sum_{\lambda\in\Lambda}\lambda^2|(u_f,\phi_\lambda)_{L^2}|^2\|\phi_\lambda\|_{L^2}^2<\infty.$$
At the moment, we are ready to introduce the pseudodifferential
operators induced by $\mathscr{H}_{\bf A}^2$.
\begin{defi} Let ${\bf A}\in (C^1(\overline{\Omega}))^N$. Assume that the complex-valued functional $F\in C(\mathbb{R}_+)$ is polynomially bounded. One defines a
generalized linear pseudodifferential operator on $\mathcal{H}_0^1$
as follows: \beq F\Big(\sqrt{\mathscr{H}_{\bf A}^2}\Big)
u(x):=\displaystyle\sum_{\lambda\in\Lambda}F(\lambda)\hat{u}(\lambda)\phi_\lambda(x).\eneq
The sequence $\{F(\lambda)\}_{\lambda^2\in\Lambda}$ is referred to
as the symbol of $F\Big(\sqrt{\mathscr{H}_{\bf A}^2}\Big)$.
\end{defi}
\begin{rem} As a matter of fact, when $f\in \mathcal{H}_0^1$, then
$$\|f\|_{\mathcal{H}_0^1}^2=\displaystyle\sum_{\lambda^2\in\Lambda}|\hat{f}(\lambda)|^2(\phi_\lambda,\phi_\lambda)_{\mathcal{H}_0^1}=\displaystyle\sum_{\lambda^2\in\Lambda}\lambda^2|\hat{f}(\lambda)|^2<\infty,$$
$$\|f\|_{L^2}^2=\displaystyle\sum_{\lambda^2\in\Lambda}|\hat{f}(\lambda)|^2(\phi_\lambda,\phi_\lambda)_{L^2}=\displaystyle\sum_{\lambda^2\in\Lambda}|\hat{f}(\lambda)|^2<\infty,$$
$$\|f\|_{\mathcal{H}^{-1}}^2=\displaystyle\sum_{\lambda^2\in\Lambda}|\hat{f}(\lambda)|^2(\phi_\lambda,\phi_\lambda)_{\mathcal{H}^{-1}}=\displaystyle\sum_{\lambda^2\in\Lambda}{1}/{\lambda^2}|\hat{f}(\lambda)|^2<\infty.$$
\end{rem}
%%%%%%%%%%%%%%%%%%%%%%%%%%%%%%%%%%%%%%%%%%%%%%%%%%%%%%%%%%%%%%%%%%%%%%%%%%%%%%%%%%%%%%%%%%%%%%%%%%%%%%
\section{Techniques of microlocal analysis and proof of Theorem 1.1}
%%%%%%%%%%%%%%%%%%%%%%%%%%%%%%%%%%%%%%%%%%%%%%%%%%%%%%%%%%%%%%%%%%%%%%%%%%%%%%%%%%%%%%%%%%%%%%%%%%%%%%
\subsection{Useful auxiliary tools from microlocal analysis}
\begin{defi}
The extended phase space is divided into the following low frequency
zone $Z_{low}(M)$, pseudodifferential zone $Z_{pd}(P,M)$ and
$p$-evolution type zone $Z_{pe}(P,M)$: ($M$ and $P$ will be given
later)
\begin{itemize}
\item $Z_{low}(M):=\Big\{(t,\lambda)\in [0,T]\times \{\lambda_1\leq\lambda\leq
M\}\Big\};$
\item $Z_{pd}(P,M):=\Big\{(t,\lambda)\in [0,T]\times \{\lambda\geq
M\}:t\lambda\leq 2^P\nu(t)\Big\};$
\item $Z_{pe}(P,M):=\Big\{(t,\lambda)\in [0,T]\times \{\lambda\geq
M\}:t\lambda\geq 2^P\nu(t)\Big\}.$
\end{itemize}
\end{defi}
\begin{defi} The following classes of symbols are defined in the
$p$-evolution type zone for large frequencies: ($\ell\in\mathbb{N},
(m_1, m_2)\in \mathbb{R}\times\mathbb{R}$)
\begin{eqnarray*}
    &&S^{\ell}\{m_{1},m_{2}\}\\
    &:=&\Big\{a\in C^\ell((0,T]; C^\infty(\lambda\geq M)):
   |D_{t}^{k}D_{\lambda}^{\alpha}a(t,\lambda)|\leq
C_{k,\alpha}\lambda^{m_{1}-\alpha}(\nu(t)/
t)^{m_{2}+k}\\
&& \mbox{for all} \,\,k,\alpha\in\mathbb{N},\ \ k\leq\ell,\,\,
(t,\lambda)\in Z_{pe}(P,M)\Big\}.
\end{eqnarray*}
\end{defi}
\begin{defi}
We define micro-energy in each zone, denoted uniformly as
$V(t,\lambda)=(V_1,V_2)^{T}$.
\begin{itemize}
\item In the low frequency zone
$Z_{low}(M)$, \beq V(t,\lambda):=(\hat{u},D_t\hat{u})^T;\eneq
\item In the pseudo-differential zone $Z_{pd}(P,M)$,
\beq V(t,\lambda):=(\lambda\hat{u},D_t\hat{u})^T;\eneq
\item In the $p$-evolution type zone $Z_{pe}(P,M)$,
\beq V(t,\lambda):=(\lambda b(t) \hat{u},D_t\hat{u})^T.\eneq
\end{itemize}
\end{defi}
\begin{defi}
$t_\lambda$ is defined as the solution of $t\lambda=2^{P}\nu(t)$,
and correspondingly, $t_\lambda\lambda=2^{P}\nu(t_\lambda)$ is
called the separating line in the time-higher frequency part.
\end{defi}
\begin{lem} According to Definition 3.2 and 3.4, in $Z_{pe}(P,M)$ we
have the following symbol calculus properties,
\begin{itemize}
  \item \quad $(\lambda b(t))^{-2}\in S^{2}\{-2,0\}$;
  \item \quad $S^r\{m_1,m_2\} \subset S^r\{m_1+ k,m_2-k\}\ for\ all\  k\in\mathbb{N};$
  \item \quad If $a \in S^r\{m_1,m_2\}$ and $b \in
  S^r\{k_1,k_2\}$, then $ab \in S^r\{m_1+k_1,m_2+k_2\}$;
  \item \quad If $a \in S^r\{m_1,m_2\}$, then $D^{k}_ta \in
  S^{r-k}\{m_1,m_2+k\}$, $D_\lambda^\alpha a \in
  S^r\{m_1-\alpha,m_2\}$;
  \item \quad If $a(t,\lambda)\in S^r\{-1,2\}$, then for $(t,\lambda)\in
Z_{pe}(P,M)$,
$|\int_{t_\lambda}^ta(\tau,\lambda)d\tau|\lesssim\nu(t_\lambda).$
\end{itemize}
\end{lem}
\begin{proof}
We only prove the last statement. Actually,
$$\Big|\int_{t_\lambda}^{t}a(\tau,\lambda)d\tau\Big|
\lesssim\int_{t_\lambda}^{t}{\nu^2(\tau)}/
{(\lambda\tau^{2})}d\tau\lesssim\nu^2(t_\lambda)/(t_\lambda\lambda)\lesssim\nu(t_\lambda).$$
\end{proof}

%%%%%%%%%%%%%%%%%%%%%%%%%%%%%%%%%%%%%%%%%%%%%%%%%%%%%%%%%%%%%%%%%%%%%%%%%%%%%%%%%%%%%%%%%%%%%%%%%%%%%%
\subsection{Estimates in $Z_{low}(M)$ and $Z_{pd}(P,M)$}
%%%%%%%%%%%%%%%%%%%%%%%%%%%%%%%%%%%%%%%%%%%%%%%%%%%%%%%%%%%%%%%%%%%%%%%%%%%%%%%%%%%%%%%%%%%%%%%%%%%%%%
For the hyperbolic magnetic Schr\"{o}dinger equation (1), the
treatments in $Z_{low}(M)$ and $Z_{pd}(P,M)$ are essentially the
same.
\begin{lem} For all $(t,\lambda)\in Z_{low}(M)$, we have the following energy
estimate:
\beq\Big|\left(\begin{array}{c}\hat{u}(t,\lambda)\\
D_t\hat{u}(t,\lambda)\end{array}\right)\Big|\lesssim|\hat{u}_0(\lambda)|+|\hat{u}_1(\lambda)|.\eneq
\end{lem}
\begin{proof}
Apply the partial Fourier transform, and we have the following
equation \beq D_t^2\hat{u}(t,\lambda)-\lambda^2
b^2(t)\hat{u}(t,\lambda)=0.\eneq Next, we study the system of first
order
\beq D_tV=\mathscr{A}(t,\lambda)V:=\left(\begin{array}{cc} 0 & 1\\
\lambda^2 b^2(t) & 0
\end{array}\right) V.\eneq
In fact, the method of successive approximation enables us to
construct the fundamental solution of the system \beq
D_{t}\mathscr{E}(t,s,\lambda)=\mathscr{A}(t,\lambda)\mathscr{E}(t,s,\lambda),\
\ \ \ \mathscr{E}(s,s,\lambda)=I.\eneq More precisely,
$\mathscr{E}(t,s,\lambda)$ is given in the form of matrizant
representation: \beq
\mathscr{E}(t,s,\lambda)=I+\sum_{k=1}^{\infty}i^{k}\int_s^t\mathscr{A}(t_{1},\lambda)
\int_s^{t_{1}}\mathscr{A}(t_{2},\lambda)\cdots\int_s^{t_{k-1}}\mathscr{A}(t_{k},\lambda)dt_{k}\cdots
dt_{1}.\eneq Actually, we have
\begin{lem} For $k\in \mathbb{N}_+$, it holds $$\Big\|\int_s^t\mathscr{A}(t_{1},\lambda)
\int_s^{t_{1}}\mathscr{A}(t_{2},\lambda)\cdots\int_s^{t_{k-1}}\mathscr{A}(t_{k},\lambda)dt_{k}\cdots
dt_{1}\Big\|\leq{1}/{k!}\Big(\int_s^t\| \mathscr{A}(r,\lambda)\|
dr\Big)^k.$$
\end{lem}
Indeed,
\begin{eqnarray*}&&\int_s^t\|\mathscr{A}(t_{1},\lambda)\|\int_s^{t_{1}}
\|\mathscr{A}(t_{2},\lambda)\|dt_2dt_1 \\
\\
&=& \int_s^t \frac{\partial}{\partial t_{1}}\Big(\int_s^{t_{1}}\|
\mathscr{A}(t_{2},\lambda)\| dt_{2}\Big)\Big(\int_s^{t_{1}}\|
\mathscr{A}(t_{2},\lambda)\| dt_2\Big)dt_1\\
\\
&=& {1}/{2}\int_s^t\frac{\partial}{\partial
t_{1}}\Big(\int_s^{t_1}\| \mathscr{A}(t_{2},\lambda)\|
dt_2\Big)^2dt_1 ={1}/{2}\Big(\int_s^t\| \mathscr{A}(r,\lambda)\|
dr\Big)^2.
\end{eqnarray*}
By induction method, the statement follows immediately.
Consequently, by applying the definition of the zone $Z_{low}(M)$,
we have
\beq\|\mathscr{E}(t,s,\lambda)\|\leq\exp\Big(\int_s^t\|\mathscr{A}(r,\lambda)\|
dr\Big)\leq\exp\Big(\int_0^{T} C(M)ds\Big)\leq C(M,T).\eneq And this
estimate leads to the final conclusion.
\end{proof}
\begin{rem}
The constant $M$ is chosen in order to separate large frequencies
from lower frequencies. Actually, we can choose any large constant
as M. As a matter of fact, the small frequencies play an
insignificant role in the discussion of loss of regularity, and it
is reasonable to neglect the influence from the electric field ${\bf
E}$.
\end{rem}
Now we sketch the discussion in $Z_{pd}(P,M)$.
\begin{lem}
For all $(t,\lambda)\in Z_{pd}(P,M)$, we have the following energy
estimate with a positive constant $c_1$ depending upon $P$ and $M$:
$$\Big|\left(\begin{array}{c}\lambda\hat{u}(t,\lambda)\\
D_t\hat{u}(t,\lambda)\end{array}\right)\Big|\lesssim\exp\left(c_1\nu(t_\lambda)\right)\left(\lambda|\hat{u}_0(\lambda)|+|\hat{u}_1(\lambda)|\right).$$
\end{lem}
\begin{proof}
We study the system of first order
$$D_tV=\mathscr{B}(t,\lambda)V:=\left(\begin{array}{cc} 0 & \lambda\\
\lambda b^2(t) & 0
\end{array}\right) V.$$
Similarly, we construct the fundamental solution of the system
$$D_{t}\mathscr{E}(t,s,\lambda)=\mathscr{B}(t,\lambda)\mathscr{E}(t,s,\lambda),\
\ \ \ \mathscr{E}(s,s,\lambda)=I.$$ More precisely,
$\mathscr{E}(t,s,\lambda)$ is given in the form of matrizant
representation
$$\mathscr{E}(t,s,\lambda)=I+\sum_{k=1}^{\infty}i^{k}\int_s^t\mathscr{B}(t_{1},\lambda)
\int_s^{t_{1}}\mathscr{B}(t_{2},\lambda)\cdots\int_s^{t_{k-1}}\mathscr{B}(t_{k},\lambda)dt_{k}\cdots
dt_{1}.$$ By the induction method as in the proof of Lemma 3.7, one
has
$$\|\mathscr{E}(t,s,\lambda)\|\leq\exp\Big(\int_s^t\|\mathscr{B}(r,\lambda)\| dr\Big)\leq\exp\Big(\int_0^{t_\lambda}
c_1\lambda ds\Big)\leq\exp\big(c_1\nu(t_\lambda)\big).$$ The final
inequality holds when we take account of the definition of
$t_\lambda$.
\end{proof}
%%%%%%%%%%%%%%%%%%%%%%%%%%%%%%%%%%%%%%%%%%%%%%%%%%%%%%%%%%%%%%%%%%%%%%%%%%%%%%%%%%%%%%%%%%%%%%%%%%%%%%
\subsection{Estimates in $Z_{pe}(P,M)$}
%%%%%%%%%%%%%%%%%%%%%%%%%%%%%%%%%%%%%%%%%%%%%%%%%%%%%%%%%%%%%%%%%%%%%%%%%%%%%%%%%%%%%%%%%%%%%%%%%%%%%%
\begin{lem}
For all $(t,\lambda)\in Z_{pe}(P,M)$, we have the following energy
estimate with a positive constant $c_1$ depending upon $P$, $M$:
\beq\Big|\left(\begin{array}{c}\lambda b(t)\hat{u}(t,\lambda)\\
D_t\hat{u}(t,\lambda)\end{array}\right)\Big|\lesssim\exp\left(c_1\nu(t_\lambda)\right)
\Big(\lambda
b(t_\lambda)|\hat{u}(t_\lambda,\lambda)|+|D_t\hat{u}(t_\lambda,\lambda)|\Big).\eneq
\end{lem}
\begin{proof}
The whole process is based on the application of two steps of
diagonalisation procedure and the construction of fundamental
solution. Taking account of the definition of micro-energy in this
zone, we study the following first order system:
\beq \begin{array}{lll}D_tV&=&\left(\begin{array}{cc} 0 & \lambda b(t)\\
\lambda b(t) & 0
\end{array}\right)V+
\left(\begin{array}{cc}D_tb(t)/b(t) & 0 \\
0 & 0
\end{array}\right) V.\end{array}\eneq
{\it Step 1: First step of diagonalisation}\\
Choose the diagonalizer:

\beq\mathscr{M}\:=\left(\begin{array}{cc} 1 & -1\\
1 & 1
\end{array}\right).\eneq
It is evident that $\mathscr{M}^{-1}$ exists. Through the transform
$V=\mathscr{M}V_{0}$, we obtain
\[
  \begin{array}{lcl}\displaystyle
0&=&D_tV_0-\mathscr{D}V_0+B_1V_0+B_2V_0\\
\\
&:=&D_tV_0-\left(\begin{array}{cc} \lambda b(t) & 0\\
0 & -\lambda b(t)  \end{array}\right) V_0\\
\\
&&-{1}/{2}\left(\begin{array}{cc} D_tb(t)/b(t)&  -D_tb(t)/b(t)\\
-D_tb(t)/b(t) &
 D_tb(t)/b(t)\\
\end{array}\right) V_0
  \end{array}
\]
where $\mathscr{D}\in S^{2}\{1,0\}$ and $B\:=
B_1+B_2\in S^{1}\{0,1\}$.\\
\\
{\it Step 2: Second step of diagonalisation}\\
\\
To carry out this step of diagonalisation, or the so-called normal
form diagonalisation, we follow the procedure of asymptotic theory
of differential equations. Namely, we construct an invertible matrix
$N_{1}(t,\xi):= I + N^{(1)}(t,\lambda)$. Define $N^{(0)}:= I,
\,B^{(0)}:= B, \,F^{(0)}:= \text{diag}(B^{(0)})$,
\[
  \begin{array}{ll}
   N_{qr}^{(1)}&:={B_{qr}^{(0)}}/({\tau_{q}-\tau_{r}}),
   q\neq r;\ N_{qq}^{(1)}:=0, \,\,\tau_{k}=(-1)^{k+1}\lambda b(t),\,\,k=1,2;\\
&\\
   B^{(1)}&:=(D_{t}-\mathscr{D}+B)(I+N^{(1)})-(I+N^{(1)})
   (D_{t}-\mathscr{D}+F^{(0)}).
  \end{array}
\]
According to the properties of the symbol calculus, $N^{(1)}\in
S^{1}\{-1,1\}$ and $F^{(0)}\in S^{1}\{0,1\}$. As for $B^{(1)}$, we
obtain the following relation: \beq
B^{(1)}=B+[N^{(1)},\mathscr{D}]-F^{(0)}+D_{t}N^{(1)}+BN^{(1)}-N^{(1)}F^{(0)}.\eneq
The construction principle implies that the sum of the first three
terms vanishes, hence $B^{(1)}\in S^{0}\{-1,2\}$. Finally, let \beq
R_{1}:=
N_{1}^{-1}B^{(1)}=N_{1}^{-1}\Big((D_{t}-\mathscr{D}+B)(I+N^{(1)})-(I+N^{(1)})
(D_{t}-\mathscr{D}+F^{(0)})\Big).\eneq By virtue of the calculus of
generalized symbols, this definition means $R_{1}\in S^{0}\{-1,2\}$.
In addition, due to the definition of $Z_{pe}(P,M)$, then
$N^{(1)}\in S^{1}\{-1,1\}$ indicates $|N_{qr}^{(1)}|\leq C/2^P$.
Consequently, an appropriate integer $P$ assures that
$\|N_{1}-I\|<1/2$ in $Z_{pe}(P,M)$, which implies the invertibility
of $N_{1}$. As a result, we have the following system after the
second step of diagonalisation:
\beq(D_t-\mathscr{D}+B)N_1=N_1(D_t-\mathscr{D}+F^{(0)}+R_1),\
\text{where}\ \ R_{1}\in S^{0}\{-1,2\}.\eneq
{\it Step 3: Estimate of the fundamental solution}\\
\\
We apply the transform $V_0=N_1V_1$ and consider the system \beq
(D_t-\mathscr{D}+F^{(0)}+R_1)V_1=0. \eneq The fundamental solution
for this equation is $\mathscr{E}=\mathscr{E}_1\mathscr{H}$, where
$\mathscr{E}_1(t,s,\lambda)$ has the following form,

\beq\left\{\begin{array}{ll}
\mathscr{E}_1(t,s,\lambda)^{(11)}=&\exp\Big(\displaystyle
i\int_{s}^{t}\lambda b(\tau)
d\tau+{1}/{2}\int_s^tb'(\tau)/b(\tau)d\tau\Big)\\
\\
\mathscr{E}_1(t,s,\lambda)^{(22)}=&\exp\Big(\displaystyle-i\int_{s}^{t}\lambda
b(\tau)
d\tau+{1}/{2}\int_s^tb'(\tau)/b(\tau)d\tau\Big)\\
\\
\mathscr{E}_1(t,s,\lambda)^{(12)}=&\mathscr{E}_1(t,s,\lambda)^{(21)}=0.
\end{array}\right.\eneq
Furthermore, $\mathscr{H}(t,s,\lambda)$ satisfies \beq
D_{t}\mathscr{H}+\underbrace{\mathscr{E}_1(s,t,\lambda)R_1(t,\lambda)\mathscr{E}_1(t,s,\lambda)}_{\widetilde{R_1}(t,s,\lambda)}\mathscr{H}=0,\
\ \mathscr{H}(s,s,\lambda)=I.\eneq Since \beq
\|\mathscr{E}_1(t,s,\lambda)\|\leq C,\eneq for all
$s,t\in[t_\lambda,T]$, then by applying the same estimation
procedure as in $Z_{pd}(P,M)$, we have \beq\|
\mathscr{H}(t,t_\lambda,\lambda)\|\leq\exp\Big({\int_{t_\lambda}^{t}\|
\widetilde{R_1}(\tau,t_\lambda,\lambda)\| d\tau}\Big)\leq
\exp(c_1\nu(t_\lambda)).\eneq Therefore, we conclude the estimate
for the fundamental solution: \beq
\|\mathscr{E}(t,t_\lambda,\lambda)\|=\|
\mathscr{E}_1\mathscr{H}\|\lesssim\exp(c_1\nu(t_\lambda)). \eneq
Using the invertibility of $\mathscr{M}$, $N_1$, and the two
transforms\beq
V_1(t,\lambda)=\mathscr{E}(t,t_{\lambda},\lambda)V_1(t_{\lambda},\lambda),\
\ \ \ V_0=N_1V_1,V=\mathscr{M}V_0,\eneq we transform
$V_1(t,\lambda)$ back to the original micro-energy $V(t,\lambda)$
and obtain \beq
\|V(t,\lambda)\|\lesssim\exp(c_1\nu(t_\lambda))\|V(t_{\lambda},\lambda)\|.
\eneq
\end{proof}
Combining the estimates in Lemma 3.6, 3.9 and 3.10, one gets the
following energy estimate in $[0,T]\times\{\lambda\geq M\}$ with a
positive constant $c_1$:
\beq\Big|\left(\begin{array}{c}\lambda\hat{u}(t,\lambda)\\
D_t\hat{u}(t,\lambda)\end{array}\right)\Big|\lesssim\exp\left(c_1\nu(t_\lambda)\right)\left(\lambda|\hat{u}_0(\lambda)|+|\hat{u}_1(\lambda)|\right).\eneq
Taking into account the definition of $t_\lambda$, generalized
pseudodifferential operators and Plancherel theorem, we arrive at
the statement of Theorem 1.1 immediately.
%%%%%%%%%%%%%%%%%%%%%%%%%%%%%%%%%%%%%%%%%%%%%%%%%%%%%%%%%%%%%%%%%%%%%%%%%%%%%%%%%%%%%%%%%%%%%%%%%%%%%%
\section{Optimality of the loss of regularity}
%%%%%%%%%%%%%%%%%%%%%%%%%%%%%%%%%%%%%%%%%%%%%%%%%%%%%%%%%%%%%%%%%%%%%%%%%%%%%%%%%%%%%%%%%%%%%%%%%%%%%%
In this section, we discuss the optimality of our estimates in
Theorem 1.1. The method of instability argument to be used was
developed in \cite{MC1} to show that a Log-type loss really appears
for hyperbolic Cauchy problems. Now we further develop this idea to
demonstrate that the precise $\nu$-loss of derivatives really
appears for the magnetic Schr\"{o}dinger equation. Let us consider
the Cauchy problem in $[0,T]\times\Omega$, ($\Omega=(0,2\pi)$, ${\bf
A}(x)=a(x)$) \beq\partial_t^2u+b^2(t)(i\partial_x+a(x))^2u=0,\ \
u(0,x)=u_0(x),\ \
\partial_tu(0,x)=u_1(x),
\eneq with $2\pi$-periodic initial Cauchy data $u_0$, $u_1$.
\begin{defi}
For a $2\pi$-periodic solution $u=u(t,x)$ in the $x$ variable, we
introduce the homogeneous energy \beq
\dot{\mathbb{E}}_s(u)(t):=\|u(t,\cdot)\|_{\dot{H}^s(\Omega)}^2+\|\partial_tu(t,\cdot)\|_{\dot{H}^{s-1}(\Omega)}^2,\
\ \ \ s\in\mathbb{R}, \eneq where $\dot{H}^s(\Omega)$ denotes the
homogeneous Sobolev space of index $s$.
\end{defi}
First we introduce some useful auxiliary functions and sequences.
\begin{defi}
For a sufficiently small $\varepsilon>0$, we define \beq
w_\varepsilon(t):=\sin
t\exp(2\varepsilon\int_0^t\psi(\tau)\sin^2\tau d\tau),\eneq \beq
a_\varepsilon(t):=1-4\varepsilon\psi(t)\sin(2t)-2\varepsilon\psi'(t)\sin^2t-4\varepsilon^2\psi^2(t)\sin^4t,\eneq
where the real non-negative $C^\infty$ function $\psi$ is
$2\pi$-periodic on $\mathbb{R}$ and identically 0 in a neighborhood
of 0. Furthermore, it satisfies \beq
\int_0^{2\pi}\psi(\tau)\sin^2(\tau)d\tau=\pi. \eneq
\end{defi}
It is easy to verify the following fact.
\begin{lem}
According to Definition 4.2, $a_\varepsilon\in C^\infty(\mathbb{R})$
and $w_\varepsilon\in C^\infty(\mathbb{R})$. Particularly,
$w_\varepsilon$ is the unique solution of following differential
equation with initial Cauchy data \beq
w''_\varepsilon(t)+a_\varepsilon(t)w_\varepsilon(t)=0,\ \
w_\varepsilon(0)=0,\ \ w'_\varepsilon(0)=1. \eneq
\end{lem}
\begin{defi}
We define a sequence of oscillating intervals $\{I_k\}_k$ by \beq
I_k:=[t_k-\rho_k/2,t_k+\rho_k/2], \eneq and a zero sequence
$\{t_k\}_k$ satisfying \beq 2^P\nu(t_k)t_k^{-1}=\lambda_k,\eneq for
each $k\in\mathbb{N}$. Furthermore, define \beq
\{\rho_k\}_k:=\Big\{2^{-P+p}\pi t_k[\nu(t_k)]/\nu(t_k)\Big\}_k,\eneq
where $p\in \mathbb{N}$ is chosen such that
$2^{p-1}\varepsilon\pi>c_1+1$.
\end{defi}
\begin{rem}
Here we only consider the case $\lim_{t\to 0}\nu(t)=+\infty$. It is
easy to check that the sequences $\{t_k\}_k,\,\{\rho_k\}_k,$\ tend
to $0$. Such choice of $\rho_k$ guarantees that $I_k$\ is contained
in $(0,T]$. Furthermore, $\lambda_k \rho_k/(4\pi)\in\mathbb{N}_+$.
\end{rem}
With these auxiliary functions and sequences, the optimality
argument can be expressed as the following theorem.
\begin{thm}
For the Cauchy problem (1.1), there exists
\begin{itemize}
\item a sequence of
coefficients $\{b^2_k(t)\}_k$ satisfying all assumptions of Theorem
1.1 with constants independent of $k$;
\item a sequence of initial Cauchy data $\{(u^k_{0}(x),u^k_{1}(x))\}_k\in
\dot{H}^s(\Omega)\times\dot{H}^{s-1}(\Omega)$;
\end{itemize}
such that the sequence of corresponding solutions $\{u^k(t,x)\}_k$
satisfies \beq\sup_k\dot{\mathbb{E}}_1(u^k)(0)\leq
C(\varepsilon),\eneq
\beq\sup_k\dot{\mathbb{E}}_1(\exp(-c_1(\varepsilon)\nu(\mu^{-1}(2^P/\sqrt{(i\nabla+{\bf
A}(x))^2})))u^k)(t)=+\infty, \text{for any}\ t\in(0,T],\eneq where
$C(\varepsilon)$ and $c_1(\varepsilon)$ depend on the sufficiently
small positive constant $\varepsilon$.
\end{thm}
\begin{proof}
We divide our proof into three steps.\\
\\
{\it Step 1: Construction of a sequence of oscillating coefficients}\\

For each $k\in\mathbb{N}$, we define the oscillating coefficient
$b^2_k(t)$ as \beq b^2_k(t):=\left\{\begin{array}{cc} 1,& t\in\
[0,T]
\setminus I_k;\\
a_\varepsilon(\lambda_k(t-t_k)),& t\in
I_k.\\
\end{array}\right. \eneq
\begin{rem}
The above definition indicates, on the one hand,  $b^2_k\in
C^\infty(\mathbb{R})$ since $a_\varepsilon$ is identically equal to
1 in a neighborhood of $I_k$. On the other hand, \beq
0<b^2_0\leq\inf_{t\in[0,T]} b^2_k(t)\leq\sup_{t \in
[0,T]}b_k^2(t)\leq b^2_1<\infty, \eneq where the positive constants
$b_0$ and $b_1$ are independent of $k$ when we choose an appropriate
$\varepsilon>0$. Simple calculations show that the coefficient
$b^2_k$ satisfies all assumptions of Theorem 1.1 in the interval
$I_k$. While in $[0,T]\backslash I_k$, it is trivial.
\end{rem}
{\it Step 2: Construction of auxiliary functions}\\

Next we study the family of Cauchy problems in
$[t_k-\rho_k/2,t_k+\rho_k/2]\times\Omega$, \beq
\partial_t^2u^k+b^2_k(t)(i\partial_x+a(x))^2u^k=0,\ u^k(t_k,x)=0,\
\partial_tu^k(t_k,x)=u^k_{1}(x).\eneq
Let the initial Cauchy data be \beq
u^k_{1}(x)=\phi_{\lambda_k}(x)\eneq and apply the coordinate
transform \beq s=\lambda_k(t-t_k).\eneq At the same time, define
\beq v^k(s,x):= u^k(t(s),x),\eneq then for
$s\in[-\lambda_k\rho_k/2,\lambda_k\rho_k/2]$, we get \beq
\partial_s^2v^k+\lambda_k^{-2}a_\varepsilon(s)(i\partial_x+a(x))^2v^k=0,\ v^k(0,x)=0,\
\partial_sv^k(0,x)=u^k_{1}(x)/\lambda_k. \eneq
As a matter of fact, we have a unique solution for (49) in the form
of \beq
v^k(s,x)=\lambda_k^{-1}\phi_{\lambda_k}(x)w_\varepsilon(s).\eneq
Transforming back to $u^k(t,x)$, we arrive at \beq
u^k(t,x)=\lambda_k^{-1}\phi_{\lambda_k}(x)w_\varepsilon(\lambda_k(t-t_k))\eneq
in $I_k$. Further calculations lead to \beq
\begin{array}{ll} u^k\big(t_k-\rho_k/2,x\big)=0,&
\partial_tu^k\big(t_k-\rho_k/2,x\big)=\phi_{\lambda_k}(x)\exp(-\varepsilon\rho_k\lambda_k/2),\\
\\
u^k\big(t_k+\rho_k/2,x\big)=0,&
\partial_tu^k\big(t_k+\rho_k/2,x\big)=\phi_{\lambda_k}(x)\exp(\varepsilon\rho_k\lambda_k/2).
\end{array} \eneq
{\it Step 3: Existence of $\nu$-loss of regularity}\\

Now we introduce an energy conservation law in the sense of
pseudo-differential operators.
\begin{lem}
For the Cauchy problem in\ $(t,x)\in \mathbb{R}\times\Omega$, \beq
\partial_t^2u+(i\partial_x+a(x))^{2}u=0,\ \ \ u(t_0,x)=0,\ \ \
\partial_tu(t_0,x)=\phi_{\lambda}(x),\eneq
then the energy conservation law holds, that is, \beq
\dot{\mathbb{E}}_s(u)(t)=\dot{\mathbb{E}}_s(u)(t_0). \eneq
\end{lem}
\begin{proof}
In effect, we have the following explicit representation of the
unique solution by virtue of separation of variables: \beq
u(t,x)=\sin(\lambda(t-t_0))\phi_\lambda (x)/\lambda. \eneq By
applying the definition of homogeneous Sobolev spaces
$\dot{H}^s(\Omega)$, $s\in\mathbb{R}$, we calculate the homogeneous
energy for the solution $u$. It holds that \beq\begin{array}{ll}
\dot{\mathbb{E}}_s(u)(t)&=\|u(t,\cdot)\|_{\dot{H}^s(\Omega)}^2+\|\partial_tu(t,\cdot)\|_{\dot{H}^{s-1}(\Omega)}^2\\
\\
&=\displaystyle\sum_{\lambda^2\in\Lambda}|\hat{u}(t,\lambda)|^2|\lambda|^{2s}+\displaystyle\sum_{\lambda^2\in\Lambda}|\partial_t
\hat{u}(t,\lambda)|^2|\lambda|^{2(s-1)}\\
\\
&=\lambda^{2(s-1)}(\sin^2(\lambda(t-t_0))+\cos^2(\lambda(t-t_0)))\\
\\
&=\lambda^{2(s-1)}=\dot{\mathbb{E}}_s(u)(t_0).
\end{array}\eneq
\end{proof}
Therefore,
\begin{eqnarray} && \dot{\mathbb{E}}_1(u^k)(t)=\exp(-\varepsilon\rho_k\lambda_k),\ \ \text{for}\ \
t\in[0,t_k-\rho_k/2]; \\
&& \dot{\mathbb{E}}_1(u^k)(t)=\exp(\varepsilon\rho_k\lambda_k),\ \
\text{for}\ \ t\in[t_k+\rho_k/2,T].
\end{eqnarray}
It is evident that (41) follows directly from (57). While for
$t\in[t_k+\rho_k/2,T]$, we have
\beq\begin{array}{ll}&\dot{\mathbb{E}}_1(\exp(-c_1\nu(\mu^{-1}(2^P/\sqrt{(i\nabla+{\bf A}(x))^2})))u^k)(t)\\
\\
=&\dot{\mathbb{E}}_1(\exp(-c_1\nu(\mu^{-1}(2^P/\lambda_k)))u^k)(t)\\
\\
=&\exp(-2c_1\nu(\mu^{-1}(2^P/\lambda_k)))\dot{\mathbb{E}}_1(u^k)(t)\\
\\
=&\exp(-2c_1\nu(\mu^{-1}(2^P/\lambda_k))+\varepsilon\rho_k\lambda_k)\\
\\
=&\exp(-2c_1\nu(t_k)+\varepsilon\rho_k\lambda_k).
\end{array}\eneq
Taking into account the choice of $\rho_k$ and $t_k$, we have (42).
This concludes our proof.
\end{proof}
\begin{rem}
Periodic functions are very useful tools in the construction of
coefficients for instability arguments. This kind of techniques is
frequently used in the discussion of Floquet theory etc.
\cite{MC4,MC5,L3, REI1}.
\end{rem}
{\bf Acknowledgment}: This project is partially supported by US Air
Force Office of Scientific Research (AFOSR FA9550-10-1-0487),
Natural Science Foundation of Jiangsu Province (BK 20130598),
National Natural Science Foundation of China (NSFC 71673043,
71273048, 71473036, 11471072), the Scientific Research Foundation
for the Returned Overseas Chinese Scholars, Fundamental Research
Funds for the Central Universities on the Field Research of
Commercialization of Marriage between China and Vietnam (No.
2014B15214). This work is also supported by Open Research Fund
Program of Jiangsu Key Laboratory of Engineering Mechanics,
Southeast University (LEM16B06).

\end{document}